\documentclass[a4paper,11pt]{amsart}

\usepackage[utf8]{inputenc}

\usepackage{pgfplots}
\pgfplotsset{compat=newest}

\usepackage{amssymb,mathtools,amsfonts,latexsym,rawfonts, mathrsfs, lscape}
\usepackage{stmaryrd,tikz,pgfplots}
\usepackage{marginnote}  
\usepackage{hyperref}
\usepackage[all]{xy}
\usepackage{amsthm,latexsym,amssymb}
\usepackage{tensor}

\usepackage{tikz-cd} 

\usepackage{multirow}

\usepackage{array, tabularx}

\usepackage{booktabs} 
\usepackage{siunitx}

\usepackage{graphicx}
\usepackage{subfigure}

\usepackage{tikz-cd}
\usepackage{tikz}
\usepackage{color}

\usepackage{dynkin-diagrams}
\usepackage{mathdots}
\usepackage{tabularx}
\usepackage{mathtools}

\usepackage[top=1.45in, bottom=0.85in, left=0.85in, right=0.85in]{geometry}

\newtheorem{theo}{Theorem}[section]
\newtheorem*{theo*}{Theorem}

\newtheorem*{lem*}{Lemma}

\newtheorem*{cor*}{Corollary}
\newtheorem{prop}{Proposition}[section]
\newtheorem*{prop*}{Proposition}

\theoremstyle{definition}

\newtheorem*{defi*}{Definition}

\newtheorem*{conj*}{Conjecture}
\theoremstyle{definition}

\newtheorem*{ex*}{Example}

\newenvironment{rmk}%
{\vskip6pt%
\noindent%
{\it Remark.}}%
{\vskip6pt}

{\vskip6pt%
\noindent%
{\it Notation.}}%
{\vskip6pt}

{\vskip6pt%
\noindent%
{\it Observation.}}%
{\vskip6pt}

\theoremstyle{plain}
\newtheorem{thmint}{Theorem}

\renewcommand{\[}{\begin{equation*}}
\renewcommand{\]}{\end{equation*}}

\def\C{\mathbb{C}}

\DeclareMathOperator\tr{tr}

\DeclareMathOperator{\im}{i}
\DeclareMathOperator{\imm}{Im}

\DeclareMathOperator\Ric{Ric}

\def \H {\mathbb H}

\def \R {\mathbb R}

\def \C {\mathbb C}

\renewcommand{\bar}{\overline}

\title{On lcK Geometry of Gauduchon Connections}

\author{Giuseppe Barbaro}
\address{Giuseppe Barbaro\newline
		\textsc{\indent Institut for Matematik, Aarhus University\newline 
			\indent 8000, Aarhus C, Denmark}}
\email{g.barbaro@math.au.dk}

\author{Alexandra Otiman}
\address{Alexandra Otiman \newline
		\textsc{\indent Institut for Matematik, Aarhus University\newline 
			\indent 8000, Aarhus C, Denmark\newline
			\indent \indent and\newline
			\indent Institute of Mathematics ``Simion Stoilow'' of the Romanian Academy\newline 
			\indent 21 Calea Grivitei Street, 010702, Bucharest, Romania}}
	\email{aiotiman@math.au.dk} 
\thanks{Both authors are supported by a DFF Sapere Aude grant ``Conformal geometry: metrics and cohomology". The first author is a member of GNSAGA of INdAM.}

\begin{document}
	
\begin{abstract}
It is not a priori clear which of the Gauduchon connections is better suited to lcK and Vaisman manifolds.
We thus investigate the geometry of these connections through Einstein problems and more general analytic and cohomological conditions on their Ricci tensors.
Our results single out the Bismut connection as the privileged one for the non-K\"ahler lcK and Vaisman geometry.
We therefore study the {second-Bismut--Einstein} equation and provide a characterization of these structures on Hopf manifolds.
\end{abstract}
	
\maketitle
	
\section{Introduction}\label{sec: intro}
On a Hermitian manifold $(M, J, g)$, the Levi-Civita connection does not necessarily preserve the complex structure, and there is no unique connection that preserves both $g$ and $J$, or, in other words, that is Hermitian. A distinguished role is played by the one-parameter family of Hermitian connections introduced by Gauduchon in \cite{gau97}, which interpolate between the Chern connection and the Bismut connection.
While all these connections coincide in the K\"ahler setting, their curvature tensors carry genuinely different information on non-K\"ahler manifolds. In this paper, we study them in the setting of locally conformally K\"ahler (lcK) geometry, with the guiding question of whether some members of the Gauduchon family are better adapted than others to lcK structures.

A natural benchmark is to study the properties of the Ricci tensors and in particular Einstein problems.
Regarding the latter, the lack of symmetries in the Gauduchon curvature tensors leads to different ways of
tracing out the Ricci curvatures, resulting in multiple Gauduchon--Einstein problems.
Furthermore, the failure of the first Bianchi identity results in the Einstein factor possibly being non-constant.
We shall also observe that, in the lcK setting, {\em first-Gauduchon--Einstein} metrics with non-identically-zero Einstein factor are globally conformally K\"ahler, as shown by Proposition \ref{prop: firstRicci}. 
Therefore, for a non-K\"ahler lcK manifold the first Einstein equation reduces to the vanishing of the first Ricci form, a problem that was studied in \cite{bo}.

Contrary to the first Ricci form, the second Ricci form $\Ric^{t,2}$ of the Gauduchon connection $\nabla^t$ does not represent a characteristic class in general and need not satisfy any analytic or cohomological condition. We thus investigate how restrictive it is to impose such conditions in the lcK setting, asking in particular when it can be closed or $dd^c$-closed and study the associated {\em second Gauduchon--Einstein equation}
\begin{equation*}
    \Ric^{t,2}=f\omega,
\end{equation*}
and to what extent it forces the lcK structure to be Vaisman or K\"ahler.

For the Chern connection, we show that these cohomological conditions on the second Ricci form are incompatible with strict lcK geometry.
Namely, in Propositions \ref{prop: dRic Ch} and \ref{prop: ddc Ric ch} we prove that if $d\Ric^{Ch,2} = 0$, or the dimension is greater than $2$ and $dd^c\Ric^{Ch,2}=0$, then the metric is K\"ahler.
On the contrary, the same cohomological conditions are satisfied by the second-Bismut--Ricci form and single it out among the Gauduchon connections, see Theorems \ref{th: dRic vais} and \ref{th: main dRic}. 

Moving to the differential constraints on the second Ricci tensors, our results keep indicating that among the Gauduchon connections the Bismut connection is the privileged one for the non-K\"ahler lcK and Vaisman geometry, whereas analogous conditions for the other connections tend to force K\"ahler rigidity.
Specifically, the compatibility between the lcK condition and the vanishing of the second-Gauduchon--Ricci form is described in Theorem \ref{th: main Ric=0} with the following result.
\begin{thmint}
    Let $(M^n,J,\omega)$ be an lcK manifold of dimension $n\geq3$, and suppose that $\omega$ is Gauduchon and $\Ric^{t,2}=0$.
    Then, if $t\neq-1$ or $2n-1+2\sqrt{n(n-1)}$ the metric is K\"ahler.\\
    If $t=-1$ the metric is Vaisman, while if $t=2n-1+2\sqrt{n(n-1)}$ it holds $dJ\theta=(1-n)\theta\wedge J\theta$.
    In both cases, $\Ric^{t,1}$ vanishes as well.
\end{thmint}
We therefore analyze more carefully the Einstein equation for the second-Bismut--Ricci form proving the following statement in Theorem \ref{thm: Bismut-flat}.
\begin{thmint}
    Let $(M, J, \omega)$ be a compact non-K\"ahler lcK manifold and assume $\omega$ is Gauduchon satisfying the {second Bismut--Einstein equation}
    \begin{equation*}
        \Ric^{B, 2}=f \omega, \qquad \text{for some}\, f \in \mathcal{C}^{\infty}(M).
    \end{equation*}
    If $\mathrm{dim}_{\mathbb{C}}M\geq 4$, then $\omega$ is Vaisman and $f \equiv 0$.   If $\mathrm{dim}_{\mathbb{C}}M \in \{2, 3\}$ and $f \geq 0$, then $\omega$ is Vaisman and $f\equiv 0$. 
\end{thmint}

Finally, we notice that second-Bismut--Einstein manifolds are {\em Einstein--Weyl} manifolds as well.
These were studied in \cite{gi, gau,gau97,iva99,pps} and many others.
In particular, in \cite{gi} Gauduchon and Ivanov proved that the only non-K\"ahler Einstein--Weyl
manifold of complex dimension $2$ is the {\em Hopf surface} with its standard Hermitian structure, while in \cite{pps} Pedersen, Poon, and Swann showed that, with some extra assumption on the regularity of the leaves induced by $\theta^\sharp,J\theta^\sharp$, it is
also possible to prove that the Einstein--Weyl manifold must be the total space of a toric holomorphic
fibration over a K\"ahler--Einstein manifold of positive curvature (emulating the structure of the
{\em Hopf manifolds}).
We therefore characterize the Hopf manifolds admitting these structures as those with diagonal holomorphic contraction with all eigenvalues of the same norm in Theorem \ref{th: hopf}.

\section*{Acknowledgements} We thank Mehdi Lejmi for carefully reading the preprint and noticing some mistakes.

\section{Preliminaries}\label{sec: prelim}
 
A complex manifold \((M, J)\) is called \textit{locally conformally K\"ahler} (lcK) if there exists a Hermitian metric \(g\) with fundamental two-form $\omega$ satisfying $d\omega = \theta \wedge \omega$, where \(\theta\) is a closed one-form called the \textit{Lee form}. Vaisman manifolds, introduced in \cite{vai} as generalized Hopf manifolds, admit an lcK metric for which the Lee form is parallel with respect to the Levi-Civita connection. Vaisman metrics satisfy
$$\|\theta\|^2\omega=\theta\wedge J\theta - dJ\theta,$$
while the lcK metrics which satisfy this relation with $\|\theta\|$ constant have to be Vaisman by \cite[Proposition 3.2.2]{i18}.

\medskip

In \cite{gau97} Gauduchon introduced a family of canonical Hermitian connections with prescribed torsion depending on a real parameter $t\in\R$.
Given a Hermitian manifold $(M,J,g)$, the {\em Gauduchon connection} of parameter $t\in\R$ is described with respect to the Levi–-Civita connection $\nabla^g$ as
\begin{equation}\label{eq: nabla t}
    \nabla^t = \nabla^g + \frac{t-1}{4}g^{-1}d^c\omega + \frac{t+1}{4}d^c\omega(\cdot,J\cdot,J\cdot.)
\end{equation}
For parameter $t=1$ we recover the {\em Chern connection} while the parameter $t=-1$ identifies the {\em Bismut connection}.
Because of the geometric relevance of the Chern and Bismut connections, we will respectively use the superscripts $^{Ch}$ and $^B$ for them and all the associated curvature tensors.
For lcK metrics, the relation \eqref{eq: nabla t} reduces to
$$\nabla^t_XY = \nabla^g_XY -\frac{t}{2}\theta(JX)JY + \frac{t-1}{4}\theta(JY)JX + \frac{t-1}{4}\omega(X,Y)J\theta^\sharp - \frac{t+1}{4}\theta(Y)X + \frac{t+1}{4}g(X,Y)\theta^\sharp $$
for any vector fields $X,Y$, where $\theta^\sharp$ is the metric dual of $\theta$.
In particular, we get that
\begin{equation}\label{eq: nabla t-Ch}
\nabla^t_XY = \nabla^{Ch}_XY + \frac{1-t}{4}\left( \theta(Y)X +J\theta(Y)JX -2J\theta(X)JY - g(X,Y)\theta^\sharp - \omega(X,Y)J\theta^\sharp \right).    
\end{equation}

For a given connection $\nabla$ we define its curvature tensor as $R^{\nabla} = [\nabla\cdot,\nabla\cdot]-\nabla_{[\cdot,\cdot]}$. With this convention, the {\em first} and {\em second Gauduchon--Ricci} curvature tensors are respectively given by the traces, with respect to a $g$-orthonormal basis $\{e_1,\ldots,e_{2n}\}$,
\begin{align*}
    \Ric^{t,1}(\cdot,\cdot) &= -\tfrac{1}{2}\sum_{i=1}^{2n}g(R^{\nabla^t}_{\cdot,\cdot}e_i,Je_i),&\Ric^{t,2}(\cdot,\cdot) &= -\tfrac{1}{2}\sum_{i=1}^{2n}g(R^{\nabla^t}_{e_i,Je_i}\cdot,\cdot).
\end{align*}
For lcK metrics, a long but straightforward computation based on \eqref{eq: nabla t-Ch} leads to the crucial relation
\begin{equation}\label{eq: Ric t,2}
    \Ric^{t,2} = \Ric^{Ch,1} +\frac{n+t-1}{2} dJ\theta + n\frac{(1-t)^2}{8}\theta\wedge J\theta +\left[ \frac{t}{2}d^*\theta + \left((n-1)\frac{t}{2} - \frac{(1-t)^2}{8}\right)\|\theta\|^2 \right]\omega .
\end{equation}
This can be specialized to the Bismut connection as
\begin{equation}\label{eq:second-bismut}
    \Ric^{B,2}=\Ric^{\mathrm{Ch}, 1}+\frac{n-2}{2}dJ\theta+\frac{n}{2}\theta\wedge J\theta -\frac{1}{2}\left(d^*\theta+n\|\theta\|^2\right)\omega.
\end{equation}
We remark that the first Gauduchon--Ricci form is not $J$-invariant for a general Hermitian metric, but it is for lcK metrics since we have the following general relation
\begin{equation}\label{eq: Ric t1 Ch1}
    \Ric^{t,1} = \Ric^{Ch,1} +\frac{1-t}{2}(n-1)dJ\theta
\end{equation}
and $dJ\theta$ is a $(1,1)$-form for lcK metrics.
Because of this relation, the forms $\Ric^{t,1}$ are all closed and give representatives of the first Chern class of $(M,J)$.
However, this also implies that for non-K\"ahler manifolds the Einstein problem stated for the first Ricci form of lcK metrics reduces to the vanishing of these, which was studied in \cite{bo}.
\begin{prop}\label{prop: firstRicci} 
    If $\omega$ is an lcK metric on a connected complex manifold $(M, J)$ such that $\Ric^{t,1}=f\omega$ for some smooth function $f$ on $M$, then either $f\equiv 0$ or $\omega$ is globally conformally K\"ahler and $f$ has a sign.
\end{prop}
\begin{proof}
    Since $d\Ric^{t,1}=0$, for any $t \in \mathbb{R}$, we have $0=d(f \omega)=(df+f\theta) \wedge \omega$.
    Since the Lefschetz operator $\omega \wedge \cdot: \Omega^1(M) \rightarrow \Omega^3(M)$ is injective, it holds $df+f\theta=0$. By \cite[Lemma 2.2.9]{i18}, either $f\equiv0$ or $\theta$ is exact, meaning that $\omega$ is globally conformally K\"ahler. 
    In the latter case, $\theta=dh$ for some function $h$, and $f=Ce^{-h}$ for some constant $C$.
\end{proof}
\noindent On the other hand, the second Ricci forms are not closed in general, not even for lcK metrics, and their closeness gives strong geometric restrictions which we study in this paper.
Moreover, we briefly show here that, unless the metric is K\"ahler, they distinguish all the connections in the Gauduchon family. 
\begin{prop}
    Let $(M,J,\omega)$ be a compact lcK manifold. If for two distinct parameters $t,s\in\mathbb R$,
    \begin{equation*}
        \Ric^{t,2}=\Ric^{s,2},
    \end{equation*}
    then $\omega$ is K\"ahler.
\end{prop}
\begin{proof} 
    Using \eqref{eq: Ric t,2} and $\Ric^{t,2}=\Ric^{s,2}$ for $s \neq t$, we obtain
    \begin{equation*}
        \frac{1}{2}dJ\theta+\frac{n(t+s-2)}{8}\theta\wedge J\theta+\left( \frac{1}{2}d^*\theta+\left(\frac{n-1}{2}-\frac{t+s-2}{8}\right) \|\theta\|^2\right) \omega=0.
    \end{equation*}
    Taking the trace with respect to $\omega$ and using that $\tr_\omega dJ\theta=-d^*\theta - (n-1)\|\theta\|^2$, we obtain
    \begin{equation*}
        \frac{n-1}{2}\left(d^*\theta+(n-1)\|\theta\|^2\right)=0,
    \end{equation*}
    which gives by integration $\theta \equiv 0$.
\end{proof}

\medskip

The {\em Weyl connection} $\nabla^W$ is defined as the unique torsion free connection that preserves the conformal structure of the metric, that is $\nabla^Wg = \theta\otimes g$.
Because of this property, on an lcK manifold, it is a complex connection, i.e. $\nabla^W J=0$.
We consider the {\em Weyl curvature form} defined by 
$$\Ric^{W,J}(\cdot,\cdot) = -\tfrac{1}{2}\sum_{i=1}^{2n}g(R^{\nabla^W}_{e_i,Je_i}\cdot,\cdot).$$
This $2$-form is not $J$-invarian in general, but it is so when $\nabla^W$ preserves $J$, therefore on lcK manifolds.
In particular, for this clas of manifold, being {\em Einstein--Weyl} means that $\Ric^{W,J}$ is a multiple of $\omega$.
Moreover, as showed in \cite[Equation~(1)]{bo} the following useful relation holds
\begin{equation}\label{Weyl-form}
    \Ric^{W, J}=\frac{n}{2}dJ\theta + \Ric^{Ch,1}(\omega).
\end{equation}
Therefore, going through the Weyl curvature, i.e. using \eqref{Weyl-form} and \cite[Theorem~1.159]{besse}, it is also possible to relate the {\em Riemannian Ricci tensor} $\Ric^g$ with the Chern--Ricci tensor.
This equation for lcK metrics reads as:
\begin{equation}\label{eq: RicCh RicLC}
{\mathrm{Ric}}^{\mathrm{Ch}, 1}(\cdot,J\cdot) + \frac{n}{2}dJ\theta(\cdot,J\cdot) =\Ric^g -\frac{1}{2}(d^*\theta + (n-1)\|\theta\|^2)g+(n-1)\nabla\theta+\frac{n-1}{2} \theta \otimes \theta.
\end{equation}

\section{Second Chern--Ricci form of lcK metrics}
 
This section is dedicate to the study of the geometric properties of the Chern connection of lcK metrics.
More specifically, we focus on the cohomological properties of the second Chern--Ricci form and the second-Chern--Einstein equation.

\subsection{Cohomological conditions on the second Chern--Ricci form}

\begin{prop}\label{prop: dRic Ch}
    Let $(M,J,\omega)$ be a compact lcK manifold. If $dRic^{Ch,2}=0$, then $\omega$ is K\"ahler.
\end{prop}

\begin{proof} 
    Replacing the parameter $t=1$ in \eqref{eq: Ric t,2} and taking the differential we see that $dRic^{Ch,2}=0$ amounts to
    $$(df+f\theta) \wedge \omega=0,$$
    where 
    $$ f=d^{*}\theta+(n-1)\|\theta\|^2. $$
    Since $\omega\wedge\cdot: \Omega^1(M) \rightarrow \Omega^3(M)$, is injective, we get $df+f\theta=0$,
    which implies by \cite[Lemma 2.2.9]{i18} that either $f\equiv0$ or $\theta$ is exact. 
    If $f\equiv0$ then by integrating it we get $\|\theta\|^2=0$ and hence $\omega$ is K\"ahler.
    On the oder hand, if $\theta=dh$ for some function $h$, then $f=Ce^{-h}$ for some constant $C\geq0$.
    We would then get
    $$\Delta h = d^*dh = f +(1-n)\|\theta\|^2 = Ce^{-h} + (1-n)\|dh\|^2.$$
    Evaluating this equation on a maximal point of $h$ we deduce that $C=0$, implying again that $\omega$ is K\"ahler.
\end{proof}

\begin{prop}\label{prop: ddc Ric ch}
    Let $(M, J, \omega)$ be a compact lcK manifold of $\mathrm{dim}_{\mathbb{C}}M \geq 3$. If $dd^cRic^{Ch, 2}=0$, then $\omega$ is K\" ahler. 
\end{prop}
\begin{proof} 
    Taking $t=1$ in \eqref{eq: Ric t,2} and using that $\Ric^{Ch, 1}$ is $dd^c$-closed, we obtain by the assumption $dd^cRic^{Ch, 2}=0$ that
    \begin{equation*}
        dd^c(f\omega)=0,
    \end{equation*}
    where $f=d^*\theta+(n-1)\|\theta\|^2$.
    We consider the Gauduchon representative $\omega_0 = e^{-h}\omega$ in the conformal class of $\omega$, for which the equation above gives $dd^c(\tilde f\omega_0)=0$, with $\tilde f=e^hf$. 
    Thus using $d\omega_0 = \theta_0 \wedge \omega_0$ we get
    \begin{equation*}\label{eq: ddc}
       dJd\tilde f - Jd \tilde f \wedge \theta_0 + d\tilde f \wedge J\theta_0 + \tilde f dJ\theta_0 + \tilde f \theta_0 \wedge J\theta_0 = 0,
    \end{equation*}
    as a consequence of $\omega_0 \wedge \cdot$ being injective since $\mathrm{dim}_{\mathbb{C}}\geq 3$.
    We now trace this equation with respect to $\omega_0$. 
    We use that $\omega_0$ is Gauduchon and hence $\tr_{\omega_0} dJ\theta_0 = -(n-1)\|\theta_0\|^2$, $ \tr_{\omega_0}(Jd\tilde f \wedge \theta_0)=-\tr_{\omega_0}(d\tilde f \wedge J\theta_0)$, and $\tr_{\omega_0} dJd\tilde f = -\Delta^{g_0} \tilde f$ thanks to \cite[Page 502]{gau84}.
    We thus obtain 
    \begin{equation*}
        -\Delta^{g_0}\tilde f + \theta_0^{\sharp}(\tilde f) - (n-2)\tilde f\|\theta_0\|^2 =0.
    \end{equation*}
    We multiply this equation by $\tilde f$ and integrate it, getting
    \[
    0= \int_M \tilde f\Delta^{g_0}\tilde f  + (n-2)\tilde f^2\|\theta_0\|^2 \,\mathrm{vol}_{g_0} = \int_M \|d\tilde f\|^2  + (n-2)\tilde f^2\|\theta_0\|^2 \,\mathrm{vol}_{g_0}
    \]
    since $\int_M \tilde f\theta_0^{\sharp}(\tilde f)\mathrm{vol}_{g_0} = \tfrac{1}{2}\int_M \tilde f^2 d^*\theta_0=0$.
    Therefore, $\tilde f$ is constant and $\tilde f\|\theta_0\|^2=0$.
    Suppose that $\theta_0=0$. Then $\theta=dh$ and we get
    $$\tilde f = e^h\left( d^*dh +(n-1)\|dh\|^2\right).$$
    Evaluating it on a maximal and a minimal point of $h$ we see that $\tilde f$ has to vanish. Therefore, $\theta=0$ and $\omega$ is K\"ahler.
\end{proof}

\subsection{second-Chern--Einstein metrics on lcK manifolds}
In order to study second-Chern--Einstein manifolds we first notice that they are Einstein--Weyl as well.
As a matter of fact, when specialized for the Chern connection, the identities \eqref{eq: Ric t,2} and \eqref{Weyl-form} give
\begin{equation*}\label{eq: Einstein Chern Weyl}
    \Ric^{Ch,2} = \Ric^{W,J} +\left( \frac{1}{2}d^*\theta + \frac{n-1}{2}\|\theta\|^2 \right)\omega .
\end{equation*}
It follows that, for lcK metrics, the second-Chern--Einstein condition and the Einstein--Weyl condition are equivalent.
Remarkably, the authors of \cite{gi,pps} proved that these metrics have to be Vaisman.

\medskip

We determine here which Hopf manifolds admit a second-Chern--Einstein metric, or equivalently an Einstein--Weyl structure. 
Before doing that, let's briefly recall that primary Hopf manifolds are complex manifolds constructed as quotients of $\mathbb{C}^n \setminus \{0\}$ by an infinite cyclic group generated by a holomorphic contraction. 
They are known to carry lcK metrics by \cite{go98, ov16, ov23, ko}, while they are Vaisman only when their fundamental group is generated by a diagonal contraction $\gamma:\mathbb{C}^n \setminus \{0\} \rightarrow \mathbb{C}^n \setminus \{0\}$, $\gamma(z_1, \ldots, z_n)=(\lambda_1 z_1, \ldots, \lambda_n z_n)$, see e.g. \cite{io25,bel00}. 
Therefore, we shall only consider diagonal contractions.
We characterize second-Chern--Einstein metrics on these manifolds as follows. 

\begin{theo}\label{th: hopf}
Let $X=(\mathbb{C}^{n}\setminus\{0\})/\langle\gamma\rangle$, $n\geq 2$, 
be a primary Hopf manifold. Then $X$ admits a Einstein--Weyl lcK metric if and only if $\gamma$ is holomorphically conjugate to a diagonal contraction
$$d_{\lambda}(z_1,\ldots,z_n)=(\lambda_1z_1,\ldots,\lambda_nz_n), \qquad 0<|\lambda_j|<1,$$
such that
$$|\lambda_1|=\cdots=|\lambda_n|.$$
\end{theo}

\begin{proof}
If the condition $|\lambda_1|=\cdots=|\lambda_n|$ is met, then a straightforward computation shows that the metric $\omega=\frac{dd^c\|z\|^2}{\|z\|^2}$ is Einstein--Weyl. 

Suppose now that $X$ admits an Einstein--Weyl lcK metric $\omega$. Then $\omega$ is Vaisman and since $\pi_1(X)=\mathbb{Z}$, the rank of its Lee form is 1. Therefore, by the structure theorem for compact Vaisman manifolds (\cite{ov03}), the universal cover with the corresponding K\" ahler metric $\tilde{\omega}$ is biholomorphic and isometric to a K\" ahler cone associated to a Sasaki manifold:
$$\bigl(\mathbb{C}^n \setminus \{0\}, \tilde{\omega}\bigr) \simeq \bigl(\mathbb C(S),g_C,J\bigr), \qquad g_C=dr^2+r^2g_S,$$
with Reeb vector field $\xi=Jr\frac{\partial}{\partial r}$. By \cite[Th\'eor\`eme 3]{gau}, the Einstein--Weyl condition gives that $g_C$ is K\" ahler Ricci flat, which is equivalent to $g_S$ being Sasaki-Einstein (see e. g. \cite{s}). Since $X$ is Vaisman, by \cite[Prop. 5.2]{io25}, the contraction
$\gamma$ is holomorphically conjugate to a diagonal contraction $d_{\lambda}$, and the Lee vector field $\theta^{\sharp}$ is given by 
\begin{equation*}
\theta^{\sharp}=c\operatorname{Re}\left(\sum_{j=1}^n\log|\lambda_j|\,z_j\frac{\partial}{\partial z_j}\right).
\end{equation*}
Moreover, by the structure theorem of Vaisman manifolds in \cite{ov03}, we obtain the Reeb vector field as $\xi=-\frac{2}{\|\theta\|^2}J\theta^{\sharp}$. Let us assume that $\omega$ has $\|\theta\|=1$, then
\begin{equation*}
    \xi=-2J\theta^{\sharp}=-2c\sum_{j=1}^n\log|\lambda_j|\left( x_j \frac{\partial}{\partial y_j}-y_j\frac{\partial}{\partial x_j} \right).
\end{equation*}
We want now to determine $c$. Since $g_C$ is Ricci flat, the Chern connection of the canonical bundle of $\mathbb{C}^n \setminus \{0\}$ is flat and since $\mathbb{C}^n \setminus \{0\}$ is simply connected, there exists a nowhere vanishing parallel section, hence a parallel holomorphic volume form $\Omega \in \Omega^{n, 0}(\mathbb{C}^n \setminus \{0\})$, unique up to scalar multiplication.
Let $\Omega=f(z)dz_1 \wedge \ldots \wedge dz_n$. The holomorphic function $f$ on $\mathbb{C}^n \setminus \{0\}$ extends to $0$ by Hartogs', since $n \geq 2$ and $f(0) \neq 0$.
We claim that $\mathcal L_{\xi}\Omega=in\Omega$. 
Indeed, using $\nabla r\frac{\partial}{\partial r}=\operatorname{Id}$, $\nabla J=0$,  $\nabla\Omega=0$ and $\Omega$ being of type $(n,0)$, we get
\begin{align}
(\mathcal{L}_{\xi}\Omega)(Y_1,\ldots,Y_n)&=\sum_{j=1}^n\Omega(Y_1,\ldots,JY_j,\ldots,Y_n)  =in\,\Omega(Y_1,\ldots,Y_n),\label{Lie}
\end{align}

By direct computation
$\mathcal{L}_{\xi}\Omega=\left[\xi(f)+2ic\left(\sum_{j=1}^n-\log|\lambda_j|\right)f\right]dz_1\wedge\cdots\wedge dz_n.$
Using now \eqref{Lie} and that the vector field $\xi$ extends with 0 in the origin and $f(0) \neq 0$ we get 
$$2c\sum_{j=1}^n-\log|\lambda_j|=n.$$

We shall use now Lichnerowicz' obstruction (\cite[Section 6]{s}) for the existence of Sasaki--Einstein metrics. It states that
that if $f$ is a nonconstant holomorphic function on a Ricci-flat Kähler cone and $\mathcal{L}_{\xi}f=i\mu f$, for some $\mu>0$, then
$\mu\geq1$. 

In our case, we apply this to each coordinate function $z_j$ and we get
$$\mathcal{L}_{\xi}z_j=i\frac{n\bigl(-\log|\lambda_j|\bigr)}
{\displaystyle\sum_{k=1}^n-\log|\lambda_k|}z_j.$$
Therefore, for every $j$, $\frac{n\bigl(-\log|\lambda_j|\bigr)}
{\sum_{k=1}^n-\log|\lambda_k|}\geq 1$, which happens if and only if $|\lambda_1|=\ldots=|\lambda_n|$.
\end{proof}

\section{Second Bismut--Ricci form of lcK metrics}
 
This section is dedicate to the study of lcK {second Bismut--Einstein} metrics. We collect our results in the following statement.

\begin{theo}\label{thm: Bismut-flat}
    Let $(M, J, \omega)$ be a compact non-K\"ahler lcK manifold and assume $\omega$ is Gauduchon satisfying the second Bismut--Einstein equation
    \begin{equation*}
        \Ric^{B, 2}=f \omega, \qquad \text{for some}\, f \in \mathcal{C}^{\infty}(M).
    \end{equation*}
    If $\mathrm{dim}_{\mathbb{C}}M\geq 4$, then $\omega$ is Vaisman and $f \equiv 0$.   If $\mathrm{dim}_{\mathbb{C}}M \in \{2, 3\}$ and $f \geq 0$, then $\omega$ is Vaisman and $f\equiv 0$. 
\end{theo}

\begin{proof} 
    We shall first replace the second Bismut--Ricci form with the first Chern--Ricci form.
    That is, using \eqref{eq:second-bismut} our Einstein equation is equivalent to
    \begin{equation}\label{eq:chern-from-bismut}
        \Ric^{\mathrm{Ch}, 1}=\left(f+\frac{n}{2}\|\theta\|^2\right)\omega-\frac{n-2}{2}dJ\theta-\frac n2\theta\wedge J\theta.
    \end{equation}
    Then we obtain an equation for the Riemannian Ricci curvature combining \eqref{eq: RicCh RicLC} with \eqref{eq:chern-from-bismut}:
    \begin{equation}\label{eq:LC-Ricci}
        \Ric^g=2(\nabla\theta)^{1,1}-(n-1)\nabla\theta+\left(f+\frac{2n-3}{2}\|\theta\|^2\right)g-\frac{2n-3}{2}\theta\otimes\theta-\frac{n-2}{2}J\theta\otimes J\theta,
    \end{equation}
    where we used \cite[Equation~(3)]{bo}:
    \begin{equation*}
        dJ\theta(\,\cdot\,,J\,\cdot\,)=2(\nabla\theta)^{1,1}-\|\theta\|^2g+\theta\otimes\theta+J\theta\otimes J\theta.
    \end{equation*}
    We now use the contracted Bianchi identity $\delta\Ric^g+\frac{1}{2}\,ds_g=0$. Pairing it  with $\theta$ and integrating, the Gauduchon condition gives
    \begin{equation}\label{eq:Bianchi-paired}
        \int_M\left\langle\Ric^g,\nabla\theta\right\rangle\,\mathrm{vol}_g=0.
    \end{equation}
    However, $\operatorname{tr}_g(\nabla\theta)=-d^*\theta=0$ and $\int_M(\nabla_{\theta^\sharp}\theta)(\theta^\sharp)\,\mathrm{vol}_g=
    \frac{1}{2}\int_M\theta^\sharp\!\left(\|\theta\|^2\right)\,\mathrm{vol}_g=\frac{1}{2}\int_M\|\theta\|^2d^*\theta\,\mathrm{vol}_g=0$. Hence, by substituting \eqref{eq:LC-Ricci} into
    \eqref{eq:Bianchi-paired} we obtain
    \begin{equation}\label{eq:Bianchi-integral}
        0=2\int_M\|(\nabla\theta)^{1,1}\|^2\,\mathrm{vol}_g-(n-1)\int_M\|\nabla\theta\|^2\,\mathrm{vol}_g-\frac{n-2}{2}\int_M (\nabla_{J\theta^\sharp}\theta)(J\theta^\sharp)\,\mathrm{vol}_g.
    \end{equation}
    We shall use now \cite[Proposition~5]{adl}, which is a consequence of $dJ\theta$ and $d^*d\omega$ being $L^2$-orthogonal.  
    In our setting, 
    it reads as
    \begin{equation*}
        0=\int_M\left[2\|(\nabla\theta)^{1,1}\|^2-n\,J\theta\bigl([\theta^\sharp,J\theta^\sharp]\bigr)\right]\,\mathrm{vol}_g.
    \end{equation*}
    Moreover, by $\left(\nabla_{\theta^{\sharp}}g\right)(J\theta^{\sharp}, J\theta^{\sharp})=0$, we get $J\theta\bigl([\theta^\sharp,J\theta^\sharp]\bigr)=\frac{1}{2}\theta^\sharp\left(\|\theta\|^2\right)-
    (\nabla_{J\theta^\sharp}\theta)(J\theta^\sharp)$ and since the integral of the first term vanishes by $d^*\theta=0$, this gives 
    \begin{equation}\label{eq:nablathetaJtheta}
        \int_M\|(\nabla\theta)^{1,1}\|^2\,\mathrm{vol}_g=-\frac{n}{2}\int_M(\nabla_{J\theta^\sharp}\theta)(J\theta^\sharp)\,\mathrm{vol}_g.
    \end{equation}
    Substituting the above equation into \eqref{eq:Bianchi-integral}, we finally obtain
    \begin{equation}\label{eq:main-integral}
    (n-1)\int_M\|\nabla\theta\|^2\,\mathrm{vol}_g=\frac{3n-2}{n}\int_M\|(\nabla\theta)^{1,1}\|^2\,\mathrm{vol}_g.
    \end{equation}
    Since $(\nabla\theta)^{1,1}$ is the $J$-invariant part of $\nabla\theta$, $\|(\nabla\theta)^{1,1}\|^2\leq\|\nabla\theta\|^2$. Assume first that $n\geq4$. Since $(n-1)\int_M\|\nabla\theta\|^2\,\mathrm{vol}_g\leq\frac{3n-2}{n}\int_M\|\nabla\theta\|^2\,\mathrm{vol}_g$ and 
    $n-1-\frac{3n-2}{n}>0$ for $n\geq4$, it follows that $\int_M\|\nabla\theta\|^2\,\mathrm{vol}_g=0$, hence $\nabla\theta=0$. 

    We now suppose that $f\geq0$, with $n\geq2$. The general Weitzenb\" ock formula for a one-form $\eta$ is 
    \begin{equation*}
        \Delta \eta = \nabla^*\nabla \eta + \iota_{\eta^\sharp}\Ric^g,
    \end{equation*}
    where $\Delta$ is the laplacian $dd^*+d^*d$ and $\nabla^*$ is the $L^2$-dual of $\nabla$. Since $\theta$ is harmonic, the Weitzenb\"ock formula reads as
    \begin{equation*}
        0=\int_M\|\nabla\theta\|^2\,\mathrm{vol}_g+\int_M\Ric^g(\theta^\sharp,\theta^\sharp)\,\mathrm{vol}_g.
    \end{equation*}
    Evaluating \eqref{eq:LC-Ricci} on $(\theta^\sharp,\theta^\sharp)$ and using \eqref{eq:nablathetaJtheta}, we obtain
    \begin{equation*}
        \int_M f\|\theta\|^2\,\mathrm{vol}_g=-\int_M\|\nabla\theta\|^2\,\mathrm{vol}_g+\frac{2}{n}
        \int_M\|(\nabla\theta)^{1,1}\|^2\,\mathrm{vol}_g.
    \end{equation*}
    This is equivalent to 
    \begin{equation}\label{eq:Weitzenbockf}
        \int_M f\|\theta\|^2\,\mathrm{vol}_g=-\frac{n-2}{n}\int_M\|(\nabla\theta)^{1,1}\|^2\,\mathrm{vol}_g-\int_M\|(\nabla\theta)^{(2,0)+(0,2)}\|^2
        \,\mathrm{vol}_g.
    \end{equation}
    If $n\geq3$, the left-hand side of \eqref{eq:Weitzenbockf} is non-negative and the right-hand side is non-positive. Therefore both sides vanish and $(\nabla\theta)^{1,1}=0$ and $(\nabla\theta)^{(2,0)+(0,2)}=0$, hence $\nabla\theta=0$.
    It thus remains to consider $n=2$. Equation \eqref{eq:Weitzenbockf} first gives $(\nabla\theta)^{(2,0)+(0,2)}=0$. Therefore, $\nabla\theta=(\nabla\theta)^{1,1}.$
    On the other hand, equation \eqref{eq:main-integral} for $n=2$ gives $\int_M\|\nabla\theta\|^2\,\mathrm{vol}_g=2\int_M\|(\nabla\theta)^{1,1}\|^2\,\mathrm{vol}_g$. Since $\|\nabla\theta\|^2=\|(\nabla \theta)^{1, 1}\|^2+\|(\nabla \theta)^{(2, 0)+(0, 2)}\|^2$, we get $\nabla \theta=0$. 
    
    Finally, the conclusion $\nabla\theta=0$ implies that the Einstein factor $f$ vanishes. Indeed, on a Vaisman manifold the Lee and anti-Lee vector fields are parallel for the Bismut connection \cite[Theorem 3.7]{av}, and hence $\Ric^{B,2}(\theta^\sharp,J\theta^\sharp)=0$.
    Since $\Ric^{B,2}=f\omega$, we obtain $f \equiv 0$. 
\end{proof}

\begin{rmk}
   Compact complex manifolds $(M, J)$ admitting lcK metrics that are Gauduchon and second-Bismut--Einstein flat, also admit {\em second-Gauduchon--Einstein} metrics for any $t\in\R$. Namely, on these manifolds, for any $t\in\R$, one can find lcK metrics $\omega_t$ that are solutions to the equation $\Ric^{t, 2}_{\omega_t}=\lambda_t \omega_t$ for some $\lambda_t \in \mathcal{C}^{\infty}(M)$.
\end{rmk}
\begin{proof}
    By Theorem \ref{thm: Bismut-flat} we get the lcK metric $\omega$ satisfying $\Ric^{B,2}=0$ is Vaisman. By performing the following change, called a $Q$-homothetic deformation (see \cite{slesar-vilcu}):
    \begin{equation*}
    \omega_a=a\omega+(a^2-a) \theta \wedge J\theta, 
    \end{equation*}
    we obtain a new Vaisman metric with Lee form $a\theta$. For $a=\frac{2(n-1)}{n}$, this is an Einstein--Weyl metric. 
    In general, if $\tilde{\omega}$ is an Einstein--Weyl metric with Lee form $\tilde{\theta}$, let us define
    \begin{equation*}
        \omega_{t}=a_t\tilde{\omega}+(a_t^2-a_t)\tilde{\theta} \wedge J\tilde{\theta}.
    \end{equation*}
    Then for $a_t=\frac{4n}{4(n-1+t)+n(t-1)^2}$, we have $\Ric^{t, 2}_{\omega_{t}}=\frac{(n-1)(t+1)^2}{8}\omega_t$, thus satisfying the second-Gauduchon--Einstein equation.
\end{proof}

\section{Second Gauduchon--Ricci forms of lcK metrics}
 
In this section we collect evidences that among the Gauduchon connections the Bismut one is the best suited for lcK and exsplecially Vaisman geometry.
First of all, it is known from \cite{av} that the very defining property of Vaisman metrics that $\nabla\theta=0$ for lcK metrics is equivalent to ask that $\nabla^B\theta=0$, while we show here the same condition stated for any other Gauduchon connection implies the K\"ahlerness of the metric. 
\begin{prop}
    Let $(M,J,\omega)$ be a lcK manifold. If $(\nabla^t\theta)^{J_+}=0$ for $t\neq-1$ then $\omega$ is K\"ahler. Moreover, $\nabla^B\theta = \nabla^g\theta$, hence $(\nabla^B\theta)^{J_+}=0$ implies Vaisman.
\end{prop}
\begin{proof}
    We want to compute the trace $\tr_g(\nabla^t\theta)$, which has to vanish by hypothesis.
    First of all, recall that $\tr_g\nabla^g\theta = -d^*\theta$.
    Then, since the difference between the Levi--Civita and $t$-Gauduchon connections is given by \eqref{eq: nabla t},
    we have to compute the trace of $d\omega(J\cdot,\theta^\sharp,\cdot)$. This tensor, for an lcK metric, is given by
    $$d\omega(J\cdot,\theta^\sharp,\cdot) = J\theta\otimes\theta - \theta\otimes\theta + \|\theta\|^2g.$$
    Putting these together we get 
    \begin{align*}
        0= \tr_g\nabla^t\theta &= \sum_{i=1}^n g(\nabla^t_{e_i}\theta^\sharp,e_i) + g(\nabla^t_{Je_i}\theta^\sharp,Je_i) \\
        &= -d^*\theta + \sum_{i=1}^n g((\nabla^t_{e_i}-\nabla^g_{e_i})\theta^\sharp,e_i) + g((\nabla^t_{Je_i}-\nabla^g_{Je_i})\theta^\sharp,Je_i) \\
        &= -d^*\theta -\frac{t+1}{2} \sum_{i=1}^n d\omega(Je_i,\theta^\sharp,e_i)\\
        &= -d^*\theta -\frac{t+1}{2}(n-1)\|\theta\|^2.
    \end{align*}
    Integrating it we get that if $t\neq-1$ the norm $\|\theta\|$ has to vanish.
    Finally, we notice that for an lcK metric $d^c\omega=-J\theta\wedge\omega$, hence $\iota_{\theta^\sharp}(d^c\omega)=0$ and $\nabla^B\theta=\nabla^g\theta$. So the last statement follows by \cite[Theorem 1]{mor17}.
\end{proof}

Now, we turn to our main topic, which is studying the properties of the second Ricci tensor of the Gauduchon connections for lcK metrics.
For Vaisman metrics, the closeness of the second Ricci tensor rules out all the Gauduchon connections which are not the Bismut.

\begin{theo}\label{th: dRic vais}
    Let $(M,J,\omega)$ be a Vaisman manifold. Then $dRic^{t,2}=0$ if and only if $t=-1$.
\end{theo}
\begin{proof}
    We know that $\|\theta\|^2\omega=-dJ\theta+\theta\wedge J\theta$ since the metric is Vaisman.
    If we replace this in \eqref{eq: Ric t,2} we get
    $$\Ric^{t,2} = \Ric^{Ch,1} +\left(\frac{n+2t-nt-1}{2} + \frac{(1-t)^2}{8}\right) dJ\theta + \left((n-1)\frac{(1+t)^2}{8}\right)\theta\wedge J\theta .$$
    Thus differentiating we have
    $$d\Ric^{t,2}= (n-1)\frac{(1+t)^2}{8}\theta\wedge dJ\theta .$$
    Therefore, one implication follows by substituting $t=-1$ in this equation. On the other hand, if $d\Ric^{t,2}=0$ for $t\neq-1$ it has to hold $0=\theta\wedge dJ\theta =-\|\theta\|^2\theta\wedge\omega$ which is in contraddiction with the Vaisman assumption.
\end{proof}

More generally, for Gauduchon lcK metrics, we obtain the following result, highlighting again the Bismut connection as adapted to this geometry.

\begin{theo}\label{th: main dRic}
    Let $(M^n,J,\omega)$ be an lcK manifold of dimension $n\geq3$, and suppose that $\omega$ is Gauduchon and $dRic^{t,2}=0$.
    Then, if $t\neq-1,2n-1\pm2\sqrt{n(n-1)}$ the metric is K\"ahler.
\end{theo}

This result is a consequence of the following two Propositions \ref{prop: 2Ricci closed} and \ref{prop: constant norm}, where we first assume that the norm of the Lee vector field is constant, and then we prove that this condition is achieved in dimension greater than $3$.

\begin{prop}\label{prop: 2Ricci closed}
    Let $(M,J,\omega)$ be a non-K\"ahler lcK manifold. Suppose that $\omega$ is Gauduchon, $\|\theta\|$ is constant and $dRic^{t,2}=0$, then there only are two possibilities 
    \begin{itemize}
        \item $t=-1$ and $\omega$ is Vaisman;
        \item $t=2n-1\pm2\sqrt{n(n-1)}$ and $dJ\theta = (1-n)\theta\wedge J\theta$. 
    \end{itemize}
    In both cases $\Ric^{t,2}=\Ric^{t,1}$, and so in particular it is cohomologous to $\Ric^{Ch,1}$ with difference given by a multiple of $dJ\theta$.
\end{prop}
\begin{proof}
    For $t=1$ the result follows from Proposition \ref{prop: dRic Ch}.
    So we can suppose $t\neq1$.
    Differentiating \eqref{eq: Ric t,2} we get
    \begin{equation*}
        0= -c_1(t)\theta\wedge dJ\theta + c_2(t)\|\theta\|^2 \theta\wedge\omega = \theta\wedge\left(-c_1(t)dJ\theta + c_2(t)\|\theta\|^2\omega\right)
    \end{equation*}
    where $c_1(t) = n\frac{(1-t)^2}{8}$ and $c_2(t) = n\frac{t}{2} - \frac{(1+t)^2}{8}$.
    This implies that there exists a function $h$ such that
    \begin{equation}\label{eq: a}
        c_1(t)dJ\theta = h\theta\wedge J\theta + c_2(t)\|\theta\|^2\omega
    \end{equation}
    By tracing \eqref{eq: a} we get
    $$-c_1(t)(n-1)\|\theta\|^2 = h\|\theta\|^2 + nc_2(t)\|\theta\|^2,$$
    while differentiating it
    $$0=-h\theta\wedge dJ\theta + c_2(t)\|\theta\|^2\theta\wedge\omega = -h\theta\wedge dJ\theta + c_1(t)\theta\wedge dJ\theta = (c_1(t)-h)\theta\wedge dJ\theta .$$
    Now, $c_1(t)=h$ if and only if $t=-1$. In this case, \eqref{eq: a} becomes
    $$dJ\theta = \theta\wedge J\theta -\|\theta\|^2\omega ,$$
    which together with the fact that $\|\theta\|$ is constant implies $\omega$ Vaisman.
    Therefore, in this case, \eqref{eq:second-bismut} and \eqref{eq: Ric t1 Ch1} give
    $$\Ric^{B,2} = \Ric^{\mathrm{Ch}, 1}+(n-1)dJ\theta = \Ric^{B,1}.$$
    On the other hand, if $t\neq -1$, it has to hold $dJ\theta = (1-n)\theta\wedge J\theta$ and $c_2(t)=0$.
    This condition implies that $Ric^{t,2}$ is cohomologous to $Ric^{Ch,1}$, with
    \begin{equation}\label{eq: Ric for c_2}
        \Ric^{t,2} = \Ric^{Ch,1} + \frac{(n-1)(1-t)}{2}dJ\theta = \Ric^{t,1}.
    \end{equation}
\end{proof}

Notice that the values $t=-1,2n-1\pm2\sqrt{n(n-1)}$ are precisely the values for which a non-K\"ahler lcK metric can satisfy $\Ric^{t,1}=\Ric^{t,2}$.
Furthermore, the following example shows that there are non-K\"ahler and non-Vaisman manifolds which are lcK with Lee form of constant norm such that $dJ\theta = (1-n)\theta\wedge J\theta$. For this class of manifolds, the condition $dRic^{t,2}=0$ is automatically satisfied when $t=2n-1\pm2\sqrt{n(n-1)}$.
\begin{ex*}
    Consider the Tricerri metric on the Inoue surface given in coordinates $z,w$ on the universal cover $\C\times\H$ as
    $$ \omega= \frac{\im}{(\imm w)^2}dw\wedge d\bar w + \im(\imm w) dz\wedge d\bar z . $$
    A straightforward computation proves that its Lee form is $\theta=\tfrac{dy}{y}$ where $w=x+\im y$. Hence,
    $$ \theta\wedge J\theta = \frac{dx\wedge dy}{y^2} = -dJ\theta.$$
    Finally, the Lee form has constant norm $\|\theta\| =1$. 
    \hfill$\blacksquare$
\end{ex*}

We now prove that, in dimension greater than $3$, the norm $\|\theta\|$ has to be constant.

\begin{prop}\label{prop: constant norm}
    Let $(M^n,J,\omega)$ be an lcK manifold of dimension $n\geq3$. Suppose $\omega$ is Gauduchon, $dRic^{t,2}=0$ and $t\neq2n-1\pm2\sqrt{n(n-1)}$, then the norm of $\theta$ is constant.
\end{prop}
\begin{proof}
    This time, differentiating \eqref{eq: Ric t,2} we get
    \begin{equation*}
        0= -c_1(t)\theta\wedge dJ\theta + c_2(t)d\|\theta\|^2 \wedge\omega + c_2(t)\|\theta\|^2 \theta\wedge\omega.
    \end{equation*}
    Now wedging it with $\theta$ we get
    $$0=c_2(t)d\|\theta\|^2\wedge\theta\wedge\omega.$$
    Since in dimension greater than $3$ wedging with $\omega$ is injective, we get that
    $d\|\theta\|^2$ is a multiple of $\theta$, hence on the locus where $\theta\neq0$ it holds $d\|\theta\|^2 = \theta^\sharp(\|\theta\|^2)\|\theta\|^{-2}\theta$.
    Then,
    $$0= \theta\wedge\left( -c_1(t)dJ\theta + c_2(t)\sigma\omega \right)$$
    where $\sigma = \theta^\sharp(\|\theta\|^2)\|\theta\|^{-2} +\|\theta\|^2 $.
    This means that there exists a function $f$ such that
    \begin{equation*}\label{eq: 1}
        c_1(t)dJ\theta = c_2(t)\sigma\omega + f\theta\wedge J\theta.
    \end{equation*}
    Tracing it we get
    $$ (1-n)c_1(t)\|\theta\|^2 = nc_2(t)\sigma + f\|\theta\|^2,$$
    while differentiating it we get
    $$0= c_2(t)(\theta^\sharp(\sigma)\|\theta\|^{-2}+\sigma)\theta\wedge\omega - f\theta\wedge dJ\theta.$$
    Therefore,
    $$ f\sigma = c_1(t)(\theta^\sharp(\sigma)\|\theta\|^{-2}+\sigma)$$
    For the sake of notation, let us call $H=\|\theta\|^2$ and $\dot{H}=\theta^\sharp(H)$, and use just $c_1,c_2$ without expliciting the dependence on t. Then the above equations give
    $$c_1\ddot{H} = (c_1-nc_2)\dot{H}^2H^{-1} - ((1+n)c_1+2nc_2)\dot{H}H - n(c_1+c_2)H^3.$$
    We can rewrite this ODE as
    $$ c_1\ddot{H}H^{-3} - 2c_1\dot{H}^2H^{-4} = -(c_1+nc_2)\dot{H}^2H^{-4} - ((1+n)c_1+2nc_2)\dot{H}H^{-2} - n(c_1+c_2).$$
    Thus, we define a new function $\varphi = \dot{H}H^{-2}$, obtaining the ODE
    $$ \dot{\varphi}H^{-1} = -\frac{c_1+nc_2}{c_1}\varphi^2 - ((1+n)c_1+2nc_2) \varphi - n(c_1+c_2).$$
    Which means that $d\varphi = P(\varphi)\theta$ for $P$ the second order polynomial on the above RHS.
    We want to prove that $P(\varphi)$ vanishes. To do so, consider a bounded primitive function
    $$\Psi(x) = \int_0^x P(y)e^{-y^2}dy.$$
    Then, $\Psi(\varphi)\star\theta$ is a continuous $(2n-1)$-form on the whole $M$. Moreover, its differential is 
    \begin{align*}
        d(\Psi(\varphi)\star\theta) =  d\Psi(\varphi)\wedge\star\theta 
        &= P(\varphi)e^{-\varphi^2}d\varphi\wedge\star\theta =  P(\varphi)^2e^{-\varphi^2}\theta\wedge\star\theta = P(\varphi)^2e^{-\varphi^2}\|\theta\|^2 \, \mathrm{vol}_g
    \end{align*}
    which is a continuous $2n$-form on the whole $M$.
    This means that we can apply Stokes theorem giving
    \begin{align*}
        0 = \int_M d (\Psi(\varphi)\star\theta) 
        = \int_M P(\varphi)^2e^{-\varphi^2}\|\theta\|^2 \, \mathrm{vol}_g.
    \end{align*}
    It follows that $P(\varphi)$, and hence $d\varphi$, vanish.
    This implies that $H=\|\theta\|^2$ is constant.
\end{proof}

We end this section by noticing that the thesis of Theorem \ref{th: main dRic} can be improved if we strength our hypothesis to $\Ric^{t,2}=0$.
Indeed, we already proved that on a compact lcK manifold $(M^n,J,\omega)$ of dimension $n\geq3$, if $\omega$ is Gauduchon and $\Ric^{t,2}=0$, (by Theorem \ref{th: main dRic}) the parameter $t$ has to be $-1,2n-1+2\sqrt{n(n-1)}$ or $2n-1-2\sqrt{n(n-1)}$.
In the latter two cases, the metric cannot be Vaisman since $dJ\theta=(1-n)\theta\wedge J\theta$, and the first Chern--Ricci form is given by \eqref{eq: Ric for c_2} as
$$\Ric^{Ch,1} = \frac{(n-1)(t-1)}{2}dJ\theta \quad\text{for }t=2n-1 \pm 2\sqrt{n(n-1)}.$$
However, $0<2n-1-2\sqrt{n(n-1)}<1$, therefore Theorem A in \cite{bo} would imply that the metric is Vaisman, giving a contraddiction.
To sum up, we have that
\begin{theo}\label{th: main Ric=0}
    Let $(M^n,J,\omega)$ be an lcK manifold of dimension $n\geq3$, and suppose that $\omega$ is Gauduchon and $\Ric^{t,2}=0$.
    Then, if $t\neq-1$ or $2n-1+2\sqrt{n(n-1)}$ the metric is K\"ahler.
\end{theo}

\end{document}